\documentclass[12pt]{article}
\usepackage{amsmath,amssymb,amsbsy,amsfonts,amsthm,latexsym,amsopn,amstext,
                                                    amsxtra,euscript,amscd}
\begin{document}

\newtheorem{theorem}{Theorem}
\newtheorem{lemma}[theorem]{Lemma}
\newtheorem{claim}[theorem]{Claim}
\newtheorem{cor}[theorem]{Corollary}
\newtheorem{prop}[theorem]{Proposition}
\newtheorem{definition}{Definition}
\newtheorem{question}[theorem]{Open Question}
\newtheorem{conj}[theorem]{Conjecture}
\newtheorem{prob}{Problem}
\newtheorem{algorithm}[theorem]{Algorithm}

\def\squareforqed{\hbox{\rlap{$\sqcap$}$\sqcup$}}
\def\qed{\ifmmode\squareforqed\else{\unskip\nobreak\hfil
\penalty50\hskip1em\null\nobreak\hfil\squareforqed
\parfillskip=0pt\finalhyphendemerits=0\endgraf}\fi}

\def\cA{{\mathcal A}}
\def\cB{{\mathcal B}}
\def\cC{{\mathcal C}}
\def\cD{{\mathcal D}}
\def\cE{{\mathcal E}}
\def\cF{{\mathcal F}}
\def\cG{{\mathcal G}}
\def\cH{{\mathcal H}}
\def\cI{{\mathcal I}}
\def\cJ{{\mathcal J}}
\def\cK{{\mathcal K}}
\def\cL{{\mathcal L}}
\def\cM{{\mathcal M}}
\def\cN{{\mathcal N}}
\def\cO{{\mathcal O}}
\def\cP{{\mathcal P}}
\def\cQ{{\mathcal Q}}
\def\cR{{\mathcal R}}
\def\cS{{\mathcal S}}
\def\cT{{\mathcal T}}
\def\cU{{\mathcal U}}
\def\cV{{\mathcal V}}
\def\cW{{\mathcal W}}
\def\cX{{\mathcal X}}
\def\cY{{\mathcal Y}}
\def\cZ{{\mathcal Z}}

\def\fI{{\mathfrak I}}
\def\fJ{{\mathfrak J}}

\def\MNL{{\mathfrak M}(N;K,L)}
\def\VNL{V_m(N;K,L)}
\def\RNL{R(N;K,L)}

\def\MNm{{\mathfrak M}_m(N;K)}
\def\VNm{V_m(N;K)}

\def\Xm{\cX_m}

\def \C {{\mathbb C}}
\def \F {{\mathbb F}}
\def \L {{\mathbb L}}
\def \K {{\mathbb K}}
\def \Q {{\mathbb Q}}
\def \Z {{\mathbb Z}}

\def\barG{\overline{\cG}}
\def\\{\cr}
\def\({\left(}
\def\){\right)}
\def\fl#1{\left\lfloor#1\right\rfloor}
\def\rf#1{\left\lceil#1\right\rceil}

\newcommand{\pfrac}[2]{{\left(\frac{#1}{#2}\right)}}

\def\rem{\mathrm{\, rem~}}

\def \Prob{{\mathrm {}}}
\def\e{\mathbf{e}}
\def\ep{{\mathbf{\,e}}_p}
\def\epp{{\mathbf{\,e}}_{p^2}}
\def\er{{\mathbf{\,e}}_r}
\def\eM{{\mathbf{\,e}}_M}
\def\eps{\varepsilon}
\def\ord{\mathrm{ord}\,}
\def\vec#1{\mathbf{#1}}

\def \li {\mathrm {li}\,}

\def\mand{\qquad\mbox{and}\qquad}

\newcommand{\comm}[1]{\marginpar{%
\vskip-\baselineskip 
\raggedright\footnotesize
\itshape\hrule\smallskip#1\par\smallskip\hrule}}

\title{Fermat quotients:\\ Exponential sums, value set and primitive roots}

\author{{\sc Igor E.~Shparlinski} \\
{Department of Computing, Macquarie University} \\
{Sydney, NSW 2109, Australia} \\
{\tt igor.shparlinski@mq.edu.au}}

\date{\today}

\maketitle

\begin{abstract} For a prime $p$ and an integer $u$ with $\gcd(u,p)=1$, we define 
Fermat  quotients by the conditions
$$
q_p(u) \equiv \frac{u^{p-1} -1}{p} \pmod p, \qquad 0 \le q_p(u) \le p-1.
$$
D.~R.~Heath-Brown has given a bound of exponential sums 
with $N$ consecutive Fermat quotients that is nontrivial for $N\ge p^{1/2+\varepsilon}$
for any fixed $\varepsilon>0$. We use a recent idea of M.~Z.~Garaev together with a form
of the large sieve inequality due to S.~Baier and L.~Zhao, 
to show that  on average over $p$
one can obtain a nontrivial estimate for much shorter 
sums starting with $N\ge p^{\varepsilon}$. We also obtain lower bounds
on the  image size of the first $N$ consecutive Fermat quotients and use 
it to prove that there is a  positive integer  $n\le p^{3/4 + o(1)}$ 
such that $q_p(n)$ is a primitive root
modulo $p$. \end{abstract}



\section{Introduction}

\subsection{Previous results}

For a prime $p$ and an integer $u$ with $\gcd(u,p)=1$ 
the {\it Fermat quotient\/} $q_p(u)$ is defined as the unique integer
with 
$$
q_p(u) \equiv \frac{u^{p-1} -1}{p} \pmod p, \qquad 0 \le q_p(u) \le p-1.
$$

Fermat quotients appear and play a very important role in a
variety of problems, see~\cite{ErnMet,Gran1,Gran2,Ihara,Len,OstShp} 
and references therein. 

In particular, Heath-Brown~\cite{H-B} has considered the exponential sums
$$
S_p(a;N) = \sum_{\substack{n=1\\gcd(n,p)=1}}^N \ep(aq_p(n)), \qquad a \in \Z,
$$
where  for an integer $r\ge 1$ and a real $z$ we define $\er(z) = \exp(2 \pi i a z/r)$, and noticed that 
using the P{\'o}lya-Vinogradov and Burgess bounds 
(see~\cite[Theorems~12.5 and~12.6]{IwKow}) leads to the 
estimate 
\begin{equation}
\label{eq:HB bound}
\max_{\gcd(a,p)=1} |S_p(a;N)| \le N^{1-1/\nu}p^{(\nu+1)/2\nu ^2+o(1)}
\end{equation} 
with any fixed integer $\nu \ge 1$ (in fact in~\cite[Theorem~2]{H-B} it is presented only 
with $\nu =2$ but the argument applies to any $\nu$, see also~\cite[Section~4]{ErnMet}). 
It is easy to see that~\eqref{eq:HB bound} is nontrivial for  $N \ge p^{1/2+\varepsilon}$
with an arbitrary fixed $\varepsilon > 0$, 
but  becomes trivial for $N \le p^{1/2}$. 
For longer intervals of length $N \ge p^{1 +\varepsilon}$, 
a nontrivial  bound of exponential sums with
linear combinations of $s\ge 1$ consecutive values 
$q_p(u),\ldots,q_p(u+s-1)$ has been given in~\cite{OstShp},
see also~\cite{COW} for a generalisation. 
Several one-dimensional and bilinear multiplicative 
character sums, as well as sums over primes, 
with Fermat quotients have recently been estimated in~\cite{Shp1,Shp2}. 
In particular, by~\cite[Theorem~3.1]{Shp1}, for any nontrivial 
multiplicative character $\eta$  modulo $p$,  we have
\begin{equation}
\label{eq:char sum bound}
\left|\sum_{\substack{n=1\\gcd(n,p)=1}}^N \eta(q_p(n))\right| \le 
 N^{1-1/\nu}p^{(5\nu+1)/4\nu ^2+o(1)}
\end{equation} 
with any fixed integer $\nu \ge 1$. 
Furthermore,  Gomez and Winterhof~\cite{GomWint} have
estimated  multiplicative 
character sums with the $s$-fold products $q_p(n+d_1)\ldots q_p(n+d_s)$ 
over intervals of length  $N \ge p^{3/2 + \varepsilon}$
for any fixed $\varepsilon>0$. In turn, the bound of~\cite{GomWint} has
been used by Aly and Winterhof~\cite{AlyWint} to study some Boolean 
functions associated with Fermat quotients. 

The image size
$$
I_p(N) = \# \{q_p(n)\ : \ 1 \le n \le N, \ \gcd(n,p) = 1\}
$$
of the first $N$ Fermat quotients, has also been studied. 
We note that since $q_p(1)=0$, the smallest $N$
with $I_p(N)>1$, that is frequently denoted by $\ell_p$, corresponds to
the smallest nonzero Fermat quotient and has been studied in a
number of works, see~\cite{BFKS,Gran1,Gran2,Ihara,Len}. 
It is also 
shown in~\cite{Shp3} that for any fixed $\varepsilon > 0$ and sufficiently 
large $p$, we have $I_p(N) = p$ for $N \ge p^{463/252+\varepsilon}$
and  $I_p(N) = p + o(p)$ for $N \ge p^{3/2+\varepsilon}$. 

Most of these above results depend on   the following
well-known property of Fermat quotients:
\begin{equation}
\label{eq:add struct}
q_p(uv) \equiv q_p(u) + q_p(v) \pmod p, 
\end{equation}
see, for example,~\cite[Equation~(3)]{ErnMet}, which is also used 
here.

\subsection{Our results}

Here we show that on average over primes $p$ one can obtain a stronger bound 
than~\eqref{eq:HB bound}, which is
nontrivial for  $N \ge p^{\varepsilon}$. This result is based on a combination 
of the idea of Heath-Brown~\cite{H-B} to interpret the sums $S_p(a;N)$ as 
multiplicative character sums with the approach of Garaev~\cite{Gar} of 
estimating the maximal value of such sums via the large sieve inequality.
However, here instead of the classical form of the large sieve inequality,
used in~\cite{Gar}, we use the version of Baier and Zhao~\cite{BaZha2}, where 
the averaging is taken over square moduli. Our proof follows quite close to that 
of~\cite[Theorem~3]{Gar} however we allow the length of the corresponding sums
to vary with the modulus. In fact our argument gives
a bound on the sums with arbitrary primitive characters modulo $p^2$, not necessary 
of order $p$ as in the case of Fermat quotients, 
see~\eqref{eq:Gen Bound} below. 

 However here  we use some ideas of~\cite{OstShp}  
to obtain new lower bounds on the image size $I_p(N)$.
More precisely, we study the values of $N$ for which $I_p(N)$ grows as a power of $p$.
So our results somewhat interpolate between those of~\cite{BFKS},
where the case $I_p(N)>1$ has been studied,  and of~\cite{Shp3} addressing 
case of very large values of $I_p(N)$. In turn these estimates are  used
to estimate the smallest primitive root in the sequence of Fermat quotients.

We note that the bound~\eqref{eq:char sum bound},
taken with a sufficiently large $\nu$, implies that there exists
a positive integer $n \le p^{5/4 + o(1)}$  such that $q_p(n)$ is a
primitive root modulo $p$. 
Here we use our lower bounds  on $I_p(N)$ and also  bounds of double 
character sums to improve this estimate (and replace exponent $5/4$ 
with $3/4$).


\subsection{Notation}

Throughout the paper,  $p$  always denotes a  prime 
number, while  $k$, $m$ and $n$ (in both the upper and
lower cases) denote positive integer 
numbers. 

The implied constants in the symbols `$O$',  `$\ll$' 
and `$\gg$'  
may occasionally depend on the integer parameter $\nu\ge 1$
and the t real parameter $\eps > 0$, 
and are absolute otherwise.  
We recall that the notations $U = O(V)$,  $U \ll V$  and $V \gg U$ are all
equivalent to the assertion that the inequality $|U|\le cV$ holds for some
constant $c>0$.

Finally, the notation $z\sim Z$ means that $z$ must satisfy the inequality
$Z< z\leq 2Z$.

\section{Preparations}

\subsection{Basics on exponential and character sums}

We recall, that for any integers $z$ and  $r \ge 1$,  we have 
the orthogonality relation
\begin{equation}
\label{eq:Orth}
\sum_{-r/2 \le b < r/2} \er(bz) = \left\{\begin{array}{ll}
r,&\quad\text{if $z\equiv 0 \pmod r$,}\\
0,&\quad\text{if $z\not\equiv 0 \pmod r$,}
\end{array}
\right.
\end{equation}
see~\cite[Section~3.1]{IwKow}. 

We also need the bound
\begin{equation}
\label{eq:Incompl}
\sum_{n=K+1}^{K+L} \er(bn) \ll  \min\left\{L, \frac{r}{|b|}\right\},
\end{equation}
which holds for any integers  $b$, $K$ and $L\ge 1$ with $0 < |b| \le r/2$,
see~\cite[Bound~(8.6)]{IwKow}.

We also refer to~\cite[Chapter~3]{IwKow} for a background on multiplicative characters.

The link between multiplicative characters and exponential sums is given 
by the following well-known identity (see~\cite[Equation~(3.12)]{IwKow}) 
involving Gauss sums
$$
\tau_r(\chi) = \sum_{v=1}^r \chi(v) \er(v)
$$
defined for a character $\chi$ modulo an integer $r\ge 1$:

\begin{lemma}
\label{lem:tau chi}
For any multiplicative character $\chi$ modulo $r$ and an integer $b$ with 
$\gcd(b,r) = 1$,  we have
$$
\chi(b) \tau_r( \overline\chi) =  \sum_{v=1}^r \overline\chi(v) \er(bv), 
$$
where $\overline\chi$ is the complex conjugate character to $\chi$. 
\end{lemma}

By~\cite[Lemma~3.1]{IwKow} we also have:

\begin{lemma}
\label{lem:tau size}
For any primitive multiplicative character $\chi$ modulo $r$ we have
$$
|\tau_r(\chi)| = r^{1/2}.
$$
\end{lemma}

As usual, we use $\mu(d)$ and $\varphi(d)$ to denote the
M{\"o}bius and the Euler functions of an integer $d \ge 1$, respectively. 
We now mention the following well-known characterisation of primitive roots
modulo $p$ which follows from the inclusion-exclusion principle and
the orthogonality property of characters 
(see, for example,~\cite[Excercise~5.14]{LN}).

\begin{lemma}
\label{lem:prim root}
For any integer $a$, we have
$$
\frac{p-1}{\varphi(p-1)}\sum_{d \mid p-1} \frac{\mu(d)}{\varphi(d)}
\sum_{\ord \eta = d} \eta(a)
= \left\{\begin{array}{ll}
1,& \text{if $a$ is a primitive root modulo $p$,}\\
0,& \text{otherwise,}
\end{array}
\right.
$$
where the inner sum is taken over all $\varphi(d)$ multiplicative characters modulo $p$ 
of order $d$. 
\end{lemma}

\subsection{Double sums of multiplicative characters}

We now recall a result of  A.~A.~Karatsuba, 
see~\cite{Kar1} or~\cite[Chapter~VIII, Problem~9]{Kar2}.

\begin{lemma}
\label{lem:double}
For any nontrivial character $\eta$ 
modulo $p$,  and arbitrary sets $\cA, \cB \subseteq \{0,1,\ldots, p-1\}$,
we   have
$$
\left|\sum_{a \in \cA} \sum_{b \in \cB}  \eta(a+b)\right| \ll 
(\# \cA)^{1-1/2\nu} \#\cB p^{1/4\nu} +  
(\# \cA)^{1-1/2\nu} (\#\cB)^{1/2} p^{1/2\nu}
$$
with any fixed integer $\nu \ge 1$. 
\end{lemma}

\subsection{Sums  $S_p(a;N)$ and multiplicative character sums}

Heath-Brown~\cite{H-B} has noticed that~\eqref{eq:add struct}
implies that the sums $S_p(a;N)$ are essentially sums
of multiplicative characters:

\begin{lemma}
\label{lem:HB red} For any integer $a$ with $\gcd(a,p)=1$ there is 
a primitive multiplicative character $\chi$ modulo $p^2$ such that 
$$
S_p(a;N) = \sum_{n=1}^N \chi(n).
$$
\end{lemma}

We note that the character $\chi$ of Lemma~\ref{lem:HB red}
is of order $p$, however we do not use this property. 

\subsection{Large sieve for square moduli}

We make use of the following version of the large sieve 
inequality for square moduli which is due to 
Baier and Zhao~\cite[Theorem~1]{BaZha2}:

\begin{lemma}
\label{lem:BZ sieve}
Let $\alpha_1, \ldots, \alpha_K$ be an arbitrary sequence
of complex numbers and let
$$
A = \sum_{k=1}^K |\alpha_n|^2 \mand
T(u) = \sum_{k=1}^K \alpha_n \exp(2 \pi i ku).
$$
Then, for any  fixed  $\eps>0$ and arbitrary $R\ge 1$, we have
$$
\sum_{1 \le r \le R} \sum_{\substack{a =1\\ \gcd(a,r) =1}}^{r^2}
\left|T(a/r^2)\right|^2 \ll
(RK)^{\eps}\( R^3+ K +\min\{KR^{1/2},K^{1/2} R^2 \} \) A.
$$
\end{lemma}

\subsection{Ratios of small height in multiplicative subgroups
of residue rings}
\label{sec:rats}

Let $\cG$ be a multiplicative subgroup of the group 
of units in the residue ring modulo an integer $m \ge 1$.
Also, for a real $Z$, 
let  $N(m,\cG, Z)$ be the number of solutions to the congruence
$$
wx\equiv y \pmod m, \qquad\text {where } 0<|x|, |y|\leq Z, \  w\in \cG.
$$
We now recall~\cite[Theorem~1]{BKS} which gives an upper bound on $N(m,\cG,Z)$.
We note that the proof given in~\cite{BKS} works
only for $Z\ge m^{1/2}$ (which is always satisfied 
in the present paper); however it is shown in~\cite{BKSc} that the result
holds without this condition too, exactly as it is formulated 
in~\cite{BKS}.

\begin{lemma}
\label{lem:bound NZ} Let $\nu\ge 1$ be a fixed
integer and let $m\to \infty$.
Assume $\#\cG = t\gg \sqrt m$. Then for any
positive number $Z$ we have
$$
N(m,\cG,Z) \le Z t^{(2\nu +1)/2\nu(\nu+1)}m^{-1/2(\nu +1) + o(1)}
+ Z^2 t^{1/\nu}m^{-1/\nu + o(1)}.
$$
\end{lemma}

\section{Main Results} 

\subsection{Bound of exponential sums on average}

Let $P$ and $N \le P^2$ be two sufficiently large positive integers.
Assume that for every $p\sim P$ we are given an integer $N_p \sim N$.

\begin{theorem}
\label{thm:Aver Exp Sum}
For every fixed integer  $\nu \ge 1$,  we have
$$
\sum_{p \sim P} \max_{\gcd(a,p)=1} |S_{p}(a,N_p)|^{2\nu}  \le  
 \(P^3+ N^\nu +\min\{N^\nu
P^{1/2},N^{\nu/2} P^2 \} \) N^{\nu} P^{o(1)}
$$
for any sequence $N_p \sim N$, 
as $P\to \infty$.
\end{theorem}

\begin{proof} We follow the ideas of Garaev~\cite[Theorem~3]{Gar}. 

Using  Lemma~\ref{lem:HB red}, for each $p\sim N$ we choose a multiplicative character $\chi_p$ modulo $p^2$
such that 
\begin{equation}
\label{eq:chip}
\max_{\gcd(a,p)=1} |S_p(a;N_p)|  = \left| \sum_{n=1}^{N_p} \chi_p(n)\right|.
\end{equation} 

Let $M = 2N$. 
Using~\eqref{eq:Orth}, for $N_p \sim N$ we write
\begin{eqnarray*}
\sum_{n=1}^{N_p} \chi_p(n) &=&
 \sum_{m=1}^M  \chi_p(m)\frac{1}{M} \sum_{n=1}^{N_p} \sum_{b=-N}^{N-1} \eM(b(m-n)) \\
&=& \frac{1}{M} \sum_{b=-N}^{N-1}   \sum_{n=1}^{N_p} \eM(-bn)
 \sum_{m=1}^M  \chi_p(m) \eM(bm). 
\end{eqnarray*}
Recalling~\eqref{eq:Incompl}, we derive
$$
\left|\sum_{n=1}^{N_p} \chi_p(n)\right| \ll
 \sum_{b=-N}^{N-1} \frac{1}{|b|+1}   
\left|\sum_{m=1}^M  \chi_p(m) \eM(bm)\right| .
$$
Therefore, writing $|b| = |b|^{(2\nu-1)/2\nu} |b|^{1/2\nu}$,
 the H{\"o}lder inequality yields the bound
$$
\left|\sum_{n=1}^{N_p} \chi_p(n)\right|^{2\nu} \ll (\log N)^{2\nu-1}  
 \sum_{b=-N}^{N-1} \frac{1}{|b|+1}   
\left|\sum_{m=1}^M  \chi_p(m) \eM(bm)\right|^{2\nu}.
$$

Thus, by~\eqref{eq:chip} we obtain
\begin{equation}
\label{eq:Up}
\sum_{p \sim P} \max_{\gcd(a,p)=1} |S_{p}(a,N_p)|^{2 \nu} \ll
(\log N)^{2\nu-1} \sum_{b=-N}^{N-1} \frac{1}{|b|+1}  U_b, 
\end{equation} 
where
$$
U_b =  \sum_{p \sim P} \left|\sum_{m=1}^M  \chi_p(m) \eM(bm)\right|^{2\nu}.
$$
We now note that 
$$
\(\sum_{m=1}^M  \chi_p(m) \eM(bm)\)^\nu = \sum_{k=1}^{K} \rho_{b,\nu}(k)  \chi_p(k) , 
$$
 where $K = M^\nu$ and 
$$
\rho_{b,\nu}(k) = \sum_{\substack{m_1,\ldots,m_\nu=1\\ m_1\ldots m_\nu = k}}^M 
\eM(b( m_1+\ldots + m_\nu)). 
$$
Using Lemma~\ref{lem:tau chi}, we  write
$$
\(\sum_{m=1}^M  \chi_p(m) \eM(bm)\)^\nu = \sum_{k=1}^{K} \rho_{b,\nu}(k)  
 \frac{1}{\tau_{p^2}( \overline\chi_p)}  \sum_{v=1}^{p^2} \overline\chi_p(v) \epp(kv).
$$
Changing the order of summation, 
by Lemma~\ref{lem:tau size} and the Cauchy inequality, we obtain,
$$ 
\left|\sum_{m=1}^M  \chi_p(m) \eM(bm)\right|^{2\nu} \le
 \sum_{v=1}^{p^2} \left| \sum_{k=1}^{K} \rho_{b,\nu}(k) \epp(kv)\right|^2.
$$
Therefore
$$
U_b =  \sum_{p \sim P}  \sum_{v=1}^{p^2} \left| \sum_{k=1}^{K} \rho_{b,\nu}(k) \epp(kv)\right|^2.
$$
Using the standard bounds on the divisor function, see~\cite[Bound~(1.81)]{IwKow}
we conclude that 
$$
|\rho_{b,\nu}(k)| \le  \sum_{m_1\ldots m_\nu = k} 1 =  k^{o(1)}
$$ 
as $k \to \infty$. 
Hence, we now derive from Lemma~\ref{lem:BZ sieve}
$$
U_b\le \(P^3+ N^\nu +\min\{N^\nu
P^{1/2},N^{\nu/2} P^2 \} \) N^{\nu} P^{o(1)} , 
$$
which after substitution in~\eqref{eq:Up} concludes the proof.  
\end{proof} 



In particular, taking a sufficiently large $\nu$, we obtain a nontrivial bound 
on average of very short sums.

\begin{cor}
\label{cor:Short} For any fixed $\varepsilon> 0$ and  $\delta >0$
there exists $\kappa> 0$ such that for  all 
$p\sim P$ except for $O(P^{1/2 + \delta})$ of them, we have
$$
 \max_{\gcd(a,p)=1} |S_{p}(a,N_p)| \le  N_p p^{-\kappa}, 
$$
for any sequence $N_p \sim N$,  with some $N \ge  P^{\varepsilon}$.
\end{cor}

\begin{proof} Taking $\nu = \rf{3/\varepsilon}$, we obtain from 
Theorem~\ref{thm:Aver Exp Sum}
$$
\sum_{p \sim P} \max_{\gcd(a,p)=1} |S_{p}(a,N_p)|^{2\nu}  \le  
N^{2\nu} P^{1/2+o(1)}.
$$
So for  $\kappa = 0.9\delta/\nu$, we see that the inequality
$$
\max_{\gcd(a,p)=1} |S_{p}(a,N_p)| >  N_p p^{-\kappa}
$$
holds for at most $P^{1/2+\nu \kappa +o(1)}  = O(P^{1/2 + \delta})$
primes $p\sim P$. 
\end{proof} 

Furthermore, taking $\nu=3$ we derive

\begin{cor}
\label{cor:long} For any fixed $\varepsilon> 0$ 
there exists  $\delta >0$ such that for  all 
$p\sim P$ except for $O(P^{1 - \delta})$ of them, we have
$$
 \max_{\gcd(a,p)=1} |S_{p}(a,N_p)| \le  N_p p^{-1/12 + \varepsilon}, 
$$
for any sequence $N_p \sim N$,  with some $N \ge  P^{5/6}$.
\end{cor}

We note that the bound of Corollary~\ref{cor:long} is stronger than that 
of~\eqref{eq:HB bound} for $N \le p^{10/11}$. 

\subsection{Image size}
\label{sec:Image}

Here we give some lower bounds on the image size $I_p(N)$.

First we consider the case of large values of $N$, 
for which use an argument of the proof of~\cite[Theorem~13]{OstShp}.

\begin{theorem}
\label{thm:Image 1}
For every $p$ and $N < p$, we have
$$
I_p(N)  \ge (1+o(1))\frac{N^2}{p(\log N)^2}.
$$
\end{theorem}

\begin{proof}
Let $Q_p(N,a)$   be the number of primes $\ell \in \{1, \ldots, N\}$
with $q_p(\ell) = a$. 
Clearly, by the prime number theorem
$$
\sum_{a=0}^{p-1} Q_p(N,a) = (1+o(1))\frac{N}{\log N}.
$$
We now use the trivial estimate 
$$
\sum_{a=0}^{p-1} Q_p(N,a)^2  \le \sum_{a=0}^{p-1} Q_p(p-1,a)^2
$$ 
and recall that by~\cite[Bound~(17)]{OstShp} the last sum is $O(p)$.

Since by the Cauchy  inequality we  have 
$$
\(\sum_{a=0}^{p-1} Q_p(N,a)\)^2 \le I_p(N)  \sum_{a=0}^{p-1} Q_p(N,a)^2, 
$$
the result now follows. 
\end{proof}

For small values of $N$ we use an argument similar to that 
of the proof of~\cite[Theorem~11]{OstShp}.

\begin{theorem}
\label{thm:Image 2}
For every $p$ and arbirary fixed $\varepsilon>0$ there exists some $\delta > 0$ such that 
for    $p^\varepsilon <N < p$ we have
$$
I_p(N)  \ge p^\delta.
$$
\end{theorem}

\begin{proof} 
Let 
$$
\cW_p(N) = \{(u,v)\ : \ 1 \le u,v \le N, \ q_p(u) = q_p(v)\}.
$$
We   see from~\eqref{eq:add struct} 
that if $(u,v) \in \cW_p(N)$ 
then for
\begin{equation}
\label{eq:wuv}
w\equiv u/v \pmod {p^2}
\end{equation}
we have
$$
q_p(w) \equiv q_p(u) - q_p(v)   \equiv 0 \pmod p.
$$ 

Since  all values of $w$ with $q_p(w)   \equiv 0 \pmod p$ and $\gcd(w,p)=1$
satisfy 
$$
w^{p-1} \equiv 1 \pmod {p^2},
$$
they are elements of the group $\cG_p$ of the
$p$th power residues modulo $p^2$. 
Thus we see from~\eqref{eq:wuv} that  
$$
\# \cW(p) \le N(p^2, \cG_p, N), 
$$
where $N(m,\cG, Z)$ is as in Section~\ref{sec:rats}.
In particular, using Lemma~\ref{lem:bound NZ}  gives 
$$\# \cW(p) \le  
N  p^{1/2\nu(\nu+1) + o(1)}
+ N^2 p^{-1/\nu + o(1)}.
$$
Taking $\nu$ as the smallest integer that  satisfies the inequality,
$$
\frac{1}{2\nu(\nu+1)} \le \frac{\varepsilon}{2},
$$
we obtain for this $\nu$, 
$$
\# \cW(p) \le  N^{3/2+o(1)} + N^2 p^{-1/\nu + o(1)}.
$$
Since
$$
\sum_{a=0}^{p-1} R_p(N,a) = N \mand 
\sum_{a=0}^{p-1} R_p(N,a)^2 = \# \cW(p),
$$
where $R_p(N,a)$ is the the number of positive integers $n \le N$
with $q_p(n) = a$, as in the proof of Theorem~\ref{thm:Image 1},
using the Cauchy inequality we derive the desired result.
 \end{proof}

\subsection{Smallest primitive roots}

Our bound on the smallest 
primitive roots among the values of the Fermat quotients
is based on the congruence~\eqref{eq:add struct}, 
Lemma~\ref{lem:double} and the results of Sections~\ref{sec:Image}

\begin{theorem}
\label{thm:Nonres}
For every $p$, there exists
$n \le p^{3/4 + o(1)}$ such that $q_p(n)$ is a 
primitive root modulo $p$. 
\end{theorem}

\begin{proof} Let us fix some $\varepsilon>0$ and put 
$$
U = \rf{p^{3/4 + \varepsilon}} \mand V = \rf{p^{\varepsilon}}.
$$ 
By Theorem~\ref{thm:Image 1} and Theorem~\ref{thm:Image 2}, we
have
$$
I_p(U) \ge p^{1/2 +\delta} \mand I_p(V) \ge  p^{\delta}, 
$$ 
respectively. 
Furthermore, let $\cU\subseteq\{1, \ldots, U\}$ and $\cV\subseteq\{1, \ldots, V\}$ 
satisfy $\# \cU = \# \{q_p(u)\ : \ u \in \cU\} = I_p(U)$ and 
$\# \cV = \# \{q_p(v)\ : \ v \in \cV\} = I_p(V)$,
respectively.

Taking a sufficiently large $\nu$ in Lemma~\ref{lem:double} we obtain
that there is some $\kappa > 0$ depending only on $\delta$ (and thus only
on $\varepsilon$) such that 
for any nontrivial character $\eta$ 
modulo $p$, we   have
\begin{equation}
\label{eq:bound}
\sum_{u \in \cU} \sum_{v \in \cV}  \eta\(q_p(u)+q_p(v)\)  \ll \# \cU  \#\cV p^{-\kappa}.
\end{equation}
Let $T$ be the number of primitive roots of the form $q_p(u)+q_p(v)$ with
$u \in \cU$ and $v \in \cV$. By Lemma~\ref{lem:prim root}, we have
\begin{eqnarray*}
T & = & \frac{p-1}{\varphi(p-1)} \sum_{u \in \cU} \sum_{v \in \cV} \sum_{d \mid p-1} \frac{\mu(d)}{\varphi(d)} \sum_{\ord \eta = d}  \eta\(q_p(u)+q_p(v)\) \\
 & = &\frac{p-1}{\varphi(p-1)} \sum_{d \mid p-1} \frac{\mu(d)}{\varphi(d)}
 \sum_{\ord \eta = d}  \sum_{u \in \cU} \sum_{v \in \cV} \eta\(q_p(u)+q_p(v)\) .
\end{eqnarray*}
Separating the contribution  $(p-1)  \# \cU  \#\cV/\varphi(p-1)$ 
of the principal character and using~\eqref{eq:bound}, we derive
$$
T = \frac{(p-1)  \# \cU  \#\cV}{\varphi(p-1)} \(1 + p^{-\kappa}\sum_{d \mid p-1}1 \).
$$
Recalling  
the well-known estimates
on the divisor  function
$$
\sum_{d \mid s}1 
 = s^{o(1)}   
$$
as $s \to \infty$,
see~\cite[Theorem~317]{HW},  we obtain that $T > 0$.
Thus there are $u \in \cU$ and $v \in \cV$ such that $q_p(u)+q_p(v)$ is a primitive 
root modulo $p$. 
Recalling~\eqref{eq:add struct} we see that for $n = uv\le UV \le  2p^{3/4 + 2\varepsilon}$ the Fermat quotient
$q_p(n)$ is a primitive root modulo $p$. 

Since $\varepsilon$ is arbitrary, the result now follows. 
\end{proof}

In particular, we see that for any $d \mid p-1$ there is a small $d$th power 
nonresidue modulo $p$ (that is, an integer which is not a dth power modulo $p$)
 among  Fermat quotients. 

\begin{cor}
\label{cor:d nonres} For every $p$ and positive integer $d \mid p-1$ there exists 
$n \le p^{3/4+ o(1)}$ such that $q_p(n)$ is a  $d$th power 
nonresidue modulo $p$.
\end{cor}

\section{Comments}
%

We remark that a full analogue of~\eqref{eq:HB bound} 
also holds for sums over shifted intervals $[L+1, L+N]$ 
uniformly over $L$. However, the method of proof of 
Theorem~\ref{thm:Aver Exp Sum} applies only to initial intervals. 

We also notice that taking $\nu$ slowly growing with $p$ 
and estimating $\rho_{b,\nu}(k)$ more carefully, as in~\cite{Gar},
one can obtain bounds on average of even shorter sums than in 
Corollary~\ref{cor:Short}. Also as in~\cite{Gar}, one can 
estimate short exponential sums of Fermat quotients 
taken over primes $\ell$ rather than over consecutive 
integers. Note that no ``individual'' bound of such sums 
is known. 

Clearly, our results apply to  sums of arbitrary primitive
characters modulo $p^2$ giving, for every fixed integer  $\nu \ge 1$,
the bound
\begin{equation}
\label{eq:Gen Bound}
\sum_{p \sim P} \max_{\chi\in \cX_p^*} \left| \sum_{n=1}^{N_p} \chi(n)\right|^{2\nu}  \le  
 \(P^3+ N^\nu +\min\{N^\nu
P^{1/2},N^{\nu/2} P^2 \} \) N^{\nu} P^{o(1)}, 
\end{equation} 
where $\cX_p^*$ is the set of primitive multiplicative characters modulo $p^2$, 
for any sequence $N_p \sim N$, as $p\to \infty$.
 Using the results of~\cite{BaZha1} one 
can also obtain similar (albeit weaker) estimates modulo $p^k$ 
with an arbitrary  fixed $k\ge 2$.

Furthermore, it has been conjectured by Zhao~\cite{Zhao} (see also~\cite{BaZha2})
that the bound of Lemma~\ref{lem:BZ sieve} holds in the form
$$
\sum_{1 \le r \le R} \sum_{\substack{a =1\\ \gcd(a,r) =1}}^{r^2}
\left|T(a/r^2)\right|^2 \ll
(RK)^{\eps}\( R^3+ K \) A
$$
(which corresponds to shape of the classical large sieve inequality). 
In this case we obtain that for any fixed $\varepsilon> 0$ 
there exists  $\delta >0$ such that for  all 
$p\sim P$ except for $O(P^{1 - \delta})$  the bound 
$$
 \max_{\gcd(a,p)=1} |S_{p}(a,N_p)| \le  N^{1/2} p^{1/3 + \varepsilon}, 
$$
holds  for any sequence $N_p \sim N$,  with some $N \ge  p$, 
which is stronger than~\eqref{eq:HB bound}. 

Finally, we notice that although Theorem~\ref{thm:Nonres}
shows the existence of small  primitive roots modulo $p$, 
the only known bound of  multiplicative character
sums~\eqref{eq:char sum bound} is nontrivial only for 
$N \ge p^{5/4+\varepsilon}$. Estimating shorter sums, either 
for every $p$
or on average over $p$,
is an interesting open question.

\section*{Acknowledgement}

The author is  grateful to  
Arne Winterhof for the careful reading of the 
manuscript and many useful discussions. 

During the preparation of this paper, the author 
was supported in part by the  Australian Research Council 
Grant~DP1092835.

\end{document}